\documentclass[10pt]{amsart}
\usepackage[utf8]{inputenc}

\usepackage{tikz}
\usepackage{pgfmath}
\usepackage{pgfplots}
\pgfplotsset{compat=1.17}
\usetikzlibrary{calc}

\usepackage{amssymb,amsthm,amsmath}
\usepackage{mathrsfs}
\usepackage{enumerate}
\usepackage{graphicx,xcolor}
\usepackage{setspace}
\usepackage[hidelinks]{hyperref}
\usepackage{cite}

\newcommand{\dd}{\mathrm{d}}
\newcommand{\E}{\mathbb{E}}

\newcommand{\R}{\mathbb{R}}
\newcommand{\C}{\mathbb{C}} 
\newcommand{\T}{\mathbb{T}}

\newcommand{\p}[1]{\mathbb{P}\left( #1 \right)}

\newcommand{\abs}[1]{\left|#1\right|}

\newcommand{\comment}[1]{}

\usepackage{xpatch}
\makeatletter
\def\thm@space@setup{%
  \thm@preskip=12pt plus 0pt minus 0pt
  \thm@postskip=0pt plus 0pt minus 0pt
}
\xpatchcmd{\proof}{6\p@\@plus6\p@\relax}{\z@skip}{}{}
\makeatother

\newtheorem{theorem}{Theorem}
\newtheorem{lemma}[theorem]{Lemma}
\newtheorem{corollary}[theorem]{Corollary}

\theoremstyle{remark}
\newtheorem{remark}[theorem]{Remark}

\theoremstyle{definition}

\title{Negative moments of Steinhaus sums}

\author{Martin Rapaport}

\author{Tomasz Tkocz}

\author{Isabella Wu}

\address{(MR, TT \& IW) Carnegie Mellon University; Pittsburgh, PA 15213, USA.}

\email{ttkocz@math.cmu.edu}

\thanks{Research supported in part by NSF grant DMS-2246484.}

\date{\today}

\begin{document}

\begin{abstract} 
We prove a sharp upper bound on negative moments of sums of independent Steinhaus random variables (that is uniform on circles in the plane). Together with the series of earlier works: K\"onig-Kwapie\'n (2001), \mbox{Baernstein~II}--Culverhouse (2002), and K\"onig (2014), this closes the investigation of sharp $L_p-L_2$ Khinchin-type inequalities for the Steinhaus sums. Incidentally, we fix a mistake in an earlier paper, as well as provide an application to sharp bounds on R\'enyi entropy.
\end{abstract}

\maketitle

\bigskip

\begin{footnotesize}

\end{footnotesize}

\bigskip

\section{Introduction and the main result}

\begin{center}
\begin{tikzpicture}[scale=1.8, >=stealth]

\coordinate (P0) at (0,0);
\node[fill=black, circle, inner sep=1pt, label=left:$0$] at (P0) {};

\def\stepsize{{1,0.7,0.4,0.65,0.25}}
\def\angles{{30,50,-25,5,-45}}

\coordinate (current) at (0,0);

\foreach \i in {1,...,5} {
    \pgfmathsetmacro{\r}{\stepsize[\i-1]}
    \pgfmathsetmacro{\a}{\angles[\i-1]}
    \pgfmathsetmacro{\dx}{\r*cos(\a)}
    \pgfmathsetmacro{\dy}{\r*sin(\a)}

    \draw[lightgray, dashed] (current) circle (\r);

    \coordinate (next) at ($(current)+(\dx,\dy)$);
    \draw[->, thick, blue!60] (current) -- (next);
    \node[fill=black, circle, inner sep=1pt, label=right:$$] at (next) {};

    \coordinate (current) at (next);
}

\end{tikzpicture}
\end{center}

How \emph{large} is the magnitude of a random walk on the plane started at the origin with independent increments chosen uniformly at random on the circles with prescribed radii? 

An answer to this vague question of course depends on how we \emph{quantify} it, whether we are interested in a typical, or limiting  behaviour, etc. One compelling choice, especially from an analytic standpoint, is to simply ask about bounds on the $L_p$-norms, which has emerged rather naturally in local theory of Banach spaces (see, e.g. \cite{LT, Pel, Woj}), complementing results in (classical) probability (see, e.g. \cite{Gu, Ka, Tutu}).

Specifically, let $a_1, a_2, \dots$ be a sequence of positive numbers and let $\xi_1, \xi_2, \dots$ be a sequence of independent identically distributed (i.i.d.) Steinhaus random variables, that is random variables uniform on the unit circle $\T = \{z \in \C, |z| = 1\}$ in the complex plane, where $|z| = \sqrt{z\bar z}$ is the usual magnitude of a complex number $z$. We form a \emph{finite} sum,
\[ 
S = \sum_{j=1}^n a_j\xi_j,
 \]
and are interested in bounds on the $L_p$-norms, $\|S\|_p = (\E|S|^p)^{1/p}$. As explained later, for the considered Steinhaus sums $S$, these are well-defined for every $p > -1$ (although, we need abuse the term ``norm'', when $p < 1$). The degenerate case $p = 0$ is understood in the usual way arising from taking the limit,
\[ 
\|S\|_0  = \E\exp(\log|S|) \  \big(= \lim_{p \downarrow 0+} \|S\|_p\big),
 \]
often referred to as the \emph{geometric mean}.

Plainly,
\[ 
\|S\|_2 = \left(\sum_{j=1}^n a_j^2\right)^{1/2} = |(a_1, \dots, a_n)|,
 \]
being the Euclidean norm $|a|$ of the coefficient vector $a=(a_1, \dots, a_n)$, is a natural reference point. 

It turns out that up to multiplicative constants \emph{all} $L_p$-norms are comparable, in that given $p > -1$, there are positive constants $A_p, B_p$ dependent only on $p$, such that for every Steinhaus sum $S$, that is for an arbitrary choice of the coefficient vector $a$ (of arbitrary length), we have
\begin{equation}\label{eq:Ap-Bp}
A_p\|S\|_2 \leq \|S\|_p \leq B_p\|S\|_2.
 \end{equation}
These types of results are often called Khinchin-type inequalities, named after \cite{Kh}.
Readily seen by the monotonicity of $(0, +\infty) \ni p \mapsto \|\cdot\|_p$, these hold with
\[ 
A_p = 1, \qquad (p \geq 2), \qquad B_p = 1, \qquad (0 < p \leq 2)
 \]
and these constants are in fact optimal, as shown by the simple example $X = \xi_1$.
 
In general, for \emph{all} positive values of $p$, \eqref{eq:Ap-Bp} follows from the classical Khinchin inequality for Rademacher sums. The latter asserts that \eqref{eq:Ap-Bp} holds (with some constants $A_p, B_p$ dependent only on $p$) for \emph{all} sums $S = \sum v_j\varepsilon_j$, where $v_1, v_2, \dots$ are vectors in a Hilbert space and $\varepsilon_1, \varepsilon_2, \dots$ are i.i.d. Rademacher random variables (random signs) taking values $\pm 1$ with probability $\frac12$ (see, e.g. Theorem 6.2.4 in \cite{Ver} for the most general result, the Khinchin-Kahane inequality treating the sequences $(v_j)$ in arbitrary Banach spaces). Conditioning on the values of $\xi_j$, we then apply this to vectors $v_j = a_j\xi_j$ in $\C$, to deduce \eqref{eq:Ap-Bp} for the Steinhaus sum $S = \sum a_j\xi_j$ which has the same distribution as $\sum (a_j\xi_j)\varepsilon_j$, thanks to the symmetry of $\xi_j$. 

There is one important nuance: for Rademacher sums, necessarily $A_p \to 0$ as $p \to 0+$ (since, for instance, $S = \varepsilon_1 + \varepsilon_2$ has an atom at $0$, thus $\|S\|_0 = 0$). This had left an interesting open question: does the same ``bad'' behaviour necessarily persist for Steinhaus sums? It was resolved independently by Favorov in \cite{Fav1} and Ullrich in \cite{Ul1} that there is a universal constant $c>0$ such that $\|S\|_0 \geq c\|S\|_2$ for every complex-valued Steinhaus sum (later extended to arbitrary Banach spaces for Steinhaus sums in \cite{Ul2}, as well as a much more general certain class of uniform distributions on compact groups in \cite{Fav4}). This nondegenerate $L_0-L_2$ comparison has had interesting applications, see \cite{Fav1, Fav2, Fav3, Ul2}; one may think of it as an AM-GM type inequality, interestingly also present in a geometric setting of log-concave distributions, as established by Lata\l a in \cite{Lat}.

When $-1 < p < 0$, perhaps even the existence of $\E|S|^p$, i.e. that this expectation is finite for an arbitrary choice of vector $a$, warrants a word of explanation. Of course in the case when specific values of the $a_j$ forbid $S=0$, we have $\E|S|^p < \infty$ for \emph{all} values of $p$. Note that for a fixed $a > 0$ and $z \in \C$, 
\[ 
|a\xi + z|^2 = a^2 + |z|^2 + 2a\text{Re}(\xi \bar z)
 \]
which by the rotational invariance of $\xi$ has the same distribution as $a^2 + |z|^2 + 2a|z|\text{Re}(\xi) = a^2 + |z|^2 + 2a|z|\cos(2\pi U) = (a^2+|z|^2)(1+\lambda \cos(2\pi U))$ with $U$ uniform on $[0,1]$ and $\lambda = \frac{2a|z|}{a^2 + |z|^2} \in [-1,1]$. Then, using independence, the finiteness of $\E|S|^p$ follows provided that $\E|a\xi+z|^p$ is finite, which needs explanation only in the case when $\lambda = \pm 1$, which in turn is equivalent to the finiteness of the integral $\int_0^1 (1+\cos(2\pi u))^{p/2} \dd u = \int_0^1 |\cos(\pi u)|^p \dd u$ which is clear for $p > -1$.
 
That \eqref{eq:Ap-Bp} holds for $-1 < p < 0$ follows from quite general results of Gorin and Favorov from \cite{GF}, see Corollary 2 therein.  

Sharp constants in inequalities of the form \eqref{eq:Ap-Bp} for various specific distributions have been extensively investigated, and we only point to \cite{BMNO, BN, Eitan, ENT-GM, ENT-Bpn, ENT-stab, FHJSZ, Haa, HNT, J, JTT, Ko, KLO, LO-best, MRTT, New, NO, NP, Ol,  Sz} for a snapshot of landmark and recent results and to, e.g. \cite{BC, CST, HT} for further references and some historical accounts. In particular, for the Steinhaus distribution, those are known for all positive values $p$. 

More specifically, from now on, for every $p \geq 0$, we let 
\begin{equation}\label{eq:def-ApBp}
A_p, B_p \ \text{be the best constants in \eqref{eq:Ap-Bp}}, 
\end{equation}
i.e. 
\[ 
A_p = \inf \|S\|_p, \qquad B_p = \sup \|S\|_p,
 \]
where the infimum and supremum are taken over all Steinhaus sums $S = \sum_{j=1}^n a_j\xi_j$, $n \geq 1$, $a_1, a_2, \dots, a_n > 0$ with $\sum_{j=1}^n a_j^2 = 1$. Depending on the context, it may be more natural to consider Steinhaus sums with \emph{complex-valued} coefficients,
\[
S = \sum_{j=1}^n z_j\xi_j, \qquad z_1, \dots, z_n \in \C.
\]
Of course that is the same level of generality, because due to the rotational invariance of the distribution of the $\xi_j$, each $z_j\xi_j$ has the same distribution as $|z_j|\xi$. 

Constants $A_p, B_p$ are known for all $p \geq 0$, 
\begin{equation}\label{eq:Ap-Bp-sharp}
A_p = \begin{cases}
\|(\xi_1+\xi_2)/\sqrt{2}\|_p, & 0 \leq p \leq p_*, \\
\|Z\|_p, & p_* \leq p \leq 2, \\
1, & p \geq 2,
\end{cases} \qquad
B_p = \begin{cases}
1, & 0 \leq p \leq 2, \\
\|Z\|_p, & p \geq 2,
\end{cases}
 \end{equation}
where $Z$ is a Gaussian random vector in $\C$ with density $\frac{1}{\pi}e^{-|z|^2}$ so that $\E|Z|^2 = 1$ (i.e. that $\xi_j$ and $Z$ have the matching covariance matrix $\frac12I_{2\times 2}$), and $p_* = 0.48..$ is uniquely determined as a solution to the equation $\|(\xi_1+\xi_2)/\sqrt{2}\|_p = \|Z\|_p$ for $0 < p < 2$. This curious behaviour of the sharp constants featuring the sharp phase transition of the extremisers at $p_*$ was conjectured by Haagerup (as early as 1977, see \cite{Pel}, p. 151). This is reminiscent of a similar behaviour for Rademacher sums established by Haagerup in his seminal work \cite{Haa}. This has also been explored for higher dimensional analogues of random vectors uniform on spheres, with natural connections to extremal volume sections of the cube \cite{CKT}, and polydisc \cite{CST}. 

For Steinhaus sums, establishing \eqref{eq:Ap-Bp-sharp} was quite a quest: Sawa in \cite{Saw} found $A_1$, Peskir in \cite{Pes} found $B_p$ for even $p$,  independently K\"onig (the original preprint supplanted by \cite{Ko}) and Culverhouse in \cite{Cul} found $B_p$ for all $p \geq 2$, K\"onig and Kwapie\'n in \cite{KK} found $A_p$ for $1 < p < 2$ (``with vigorous hard analysis'', as put in \cite{BC}), and K\"onig in \cite{Ko} found $A_p$ for $0 < p < 1$, although there is a gap in a part of his proof, which we address here.



To bring the status of our knowledge for Steinhaus sums to the same level as for Rademacher sums, what has been left open is therefore determining the value of $A_p$ for $-1 < p < 0$ (see also Table 1 in \cite{CST}). This is our main result.

\begin{theorem}\label{thm:cp}
With notation \eqref{eq:Ap-Bp-sharp}, for $-1 < p < 0$, we have $A_p = \|(\xi_1+\xi_2)/\sqrt{2}\|_p$.
\end{theorem}

The next section is devoted to the proof of Theorem \ref{thm:cp}. We start with an overview, then work through technical lemmas, and finish off with an inductive proof. Adapting that inductive scheme, incidentally, we fix the gap from \cite{Ko} in Section \ref{sec:Ko}. Finally, Section \ref{sec:Renyi} is concerned with a sharp upper bound on the R\'enyi entropies of Steinhaus sums.

\section{Proof of Theorem \ref{thm:cp}}

\subsection{Overview.}
To be consistent with the notation used in the earlier works that we shall rely on (mainly \cite{CST}), we let $0 < p < 1$ throughout, naturally identify $\C$ with $\R^2$ (equipped with their Euclidean structures) whenever needed, and begin with rephrasing Theorem \ref{thm:cp} equivalently as: for all $n \geq 1$, $z_1, \dots, z_n \in \C$, we have
\begin{equation}\label{eq:thm-Ap-equiv}
\E\left[\left|\sum_{j=1}^{n} z_j\xi_j\right|^{-p}\ \right] \leq C_p\left(\sum_{j=1}^n |z_j|^2\right)^{-p/2},
\end{equation} 
where
\begin{equation}\label{eq:Cp}
C_p = A_p^{-p} = \E\left[\left|(\xi_1+\xi_2)/\sqrt{2}\right|^{-p}\right] = 2^{p/2}\frac{\Gamma(1-p)}{\Gamma\left(1-p/2\right)^2}.
\end{equation}
The explicit expression for $C_p$ is justified in Corollary 14 in \cite{CST} specialised to $d=2$ therein. 

Note that \eqref{eq:thm-Ap-equiv} clearly holds when $n = 1$ because $A_p = \|(\xi_1+\xi_2)/\sqrt{2}\|_p \leq  \|(\xi_1+\xi_2)/\sqrt{2}\|_2 = 1$, that is $C_p \geq 1$. Therefore throughout the whole proof, we can assume that $n \geq 2$ and that at least two of the $z_j$ are nonzero.

For the proof, in a great nutshell, we follow a well-established strategy, used e.g. in \cite{Ball-cube, CGT, CKT, CST, Haa, Ko, KK, NP, OP}.

\emph{Step I: ``Fourier+H\"older''.} We start off with a Fourier-analytic formula for $p$-moments to which we apply  H\"older's inequality for carefully chosen exponents, in order to \emph{decouple} the contributions coming from the summands $z_j\xi_j$. This powerful idea is rooted in \cite{Ball-cube, Haa, Hens}, also widely used across analysis and related areas, for instance in Section 5 of \cite{Pierce}.

Squeezing as much as possible out of this, alas, will still leave us with the case when essentially one term dominates the rest (intuitively, a factorisation--decoupling-type of bound works in a \emph{balanced} situation, when none of the terms prevails). 

\emph{Step II: Induction.} To circumvent this, we employ an inductive argument of Nazarov-Podkorytov from \cite{NP}, which they developed for random signs. It does however get nuanced for continuous distributions, and this is the part where \cite{Ko} left a rather major gap (which we shall comment on and address carefully at the end of our proof). We shall employ the advancements made in \cite{CGT, CKT, CST}.

\subsection{Specifics}
Step I begins with a Fourier-inversion formula for the power function $t^{p}$ -- for details see, e.g. Corollary 5 from \cite{CST} (applied to $X_k = z_k\xi_k$ and $d=2$) which gives
\begin{equation}\label{eq:neg-mom-formula}
\E\left[\left|\sum_{j=1}^{n} z_j\xi_j\right|^{-p}\ \right] = \kappa_{p}\int_0^\infty \prod_{j=1}^{n} J_0(t|z_j|)t^{p-1}\dd t,
 \end{equation}
provided the integral on the right hand side converges,
where
\begin{equation}\label{eq:def-kappa}
\kappa_{p} = 2^{1-p}\frac{\Gamma(1-p/2)}{\Gamma(p/2)}
 \end{equation}
and $J_0$ is the Bessel function of the first kind with index $0$, which can be defined by several means, e.g.,
\[ 
J_0(t) = \frac{1}{\pi}\int_0^\pi \cos(t\sin u) \dd u = \sum_{k=0}^{\infty} \frac{(-1)^k}{(2^kk!)^2}t^{2k}, \qquad t \in \R.
 \]
To quickly justify the integrability, we can invoke the following bound (useful in the sequel as well),
\begin{equation}\label{eq:J0-up-bd1}
|J_0(t)| \leq \min\left\{1, \sqrt{\frac{2}{\pi t}}\right\}, \qquad t \geq 0,
\end{equation}
where $|J_0(t)| \leq 1$ is clear from the integral representation and the $O(t^{-1/2})$ bound follows from a succinct general theorem on oscillating solutions of second order ODEs, Sonin's theorem (see Section 7.31 and (7.31.5) in \cite{Sze}). In particular, with at least two nonzero coefficients $z_j$, the integrand on the right hand side of \eqref{eq:neg-mom-formula} is $O(t^{p-2})$ as $t \to \infty$, whereas it plainly is $O(t^{p-1})$ as $t \to 0$, hence the integral absolutely converges. 

Next, we recall how the argument from Step I proceeds: we put
\[ 
a_j = |z_j|, \qquad j = 1, \dots, n,
 \]
assume that $n \geq 2$ and $a_j > 0$ for each $j \leq n$. With the normalisation
\[
\sum_{j=1}^n a_j^2 = 1,\]
H\"older's inequality yields
\[ 
\int_0^\infty \prod_{j=1}^{n} J_0(ta_j)t^{p-1}\dd t \leq \prod_{j=1}^n \left(\int_0^\infty J_0(ta_j)t^{p-1}\dd t\right)^{a_j^2}.
 \]
We emphasise that when $n=2$ and $a_1 = a_2 = \frac{1}{\sqrt2}$, there is equality.
Thus we define,
\begin{equation}\label{eq:def-Psi}
\Psi_p(s) = \int_0^\infty |J_0(t/\sqrt{s})|^st^{p-1} \dd t = s^{p/2}\int_0^\infty |J_0(t)|^s t^{p-1} \dd t,
\end{equation}
which, in view of \eqref{eq:J0-up-bd1}, converges provided that $s > 2p$. Consequently, Fourier-analytic formula \eqref{eq:neg-mom-formula} gives
\[ 
\E\left[\left|\sum_{j=1}^{n} z_j\xi_j\right|^{-p}\ \right] \leq \kappa_p\prod_{j=1}^n \Big(\Psi_p(a_j^{-2})\Big)^{a_j^2}
 \]
with equality when $n=2$ and $a_1 = a_2 = \frac{1}{\sqrt2}$, which becomes
\[ 
C_p = \kappa_p\Psi_p(2) = \kappa_p\int_0^\infty J_0(t)^2t^{p-1} \dd t.
 \]
Let us record here for future use that (recalling \eqref{eq:Cp}, \eqref{eq:def-kappa}),
\begin{equation}\label{eq:Psi(2)}
\Psi_p(2) = C_p/\kappa_p = 2^{3p/2-1}\Gamma(1-p)\Gamma(p/2)\Gamma(1-p/2)^{-3}.
\end{equation}
The heavy-lifting is done in the following elementary yet involved analytic lemma, which will be proved in the sequel.

\begin{lemma}\label{lm:Psi<Psi(2)}
Let $0 < p < 1$. For function $\Psi_p$ defined in \eqref{eq:def-Psi}, we have
\[ 
\Psi_p(s) \leq \Psi_p(2), \qquad s \geq 2.
 \]
\end{lemma}

In particular, if $a_j \leq \frac{1}{\sqrt2}$ for each $j \leq n$, we obtain
\[ 
\E\left[\left|\sum_{j=1}^{n} z_j\xi_j\right|^{-p}\ \right] \leq \kappa_p \prod_{j=1}^n \Big(\Psi_p(2)\Big)^{a_j^2} = \kappa_p\Psi_p(2) = C_p.
 \]
We immediately conclude with the following corollary.

\begin{corollary}\label{cor:StepI}
Let $0 < p < 1$. Inequality \eqref{eq:thm-Ap-equiv} holds provided that for each $k \leq n$,
\[ 
|z_k| \leq \frac{1}{\sqrt{2}}\left(\sum_{j=1}^n |z_j|^2\right)^{1/2}.
 \]
\end{corollary}

\begin{remark}
Numerical simulations strongly suggest that $s \mapsto \Psi_p(s)$ is in fact strictly decreasing in an open neighbourhood of $s=2$. In particular, we anticipate that the lemma fails for $s < 2$, which is precisely what we alluded to when we said that we would ''squeeze as much as possible'' in this step.
\end{remark}

\subsection{Induction}
Leveraging the homogeneity of \eqref{eq:thm-Ap-equiv} again, note that under the normalisation
\[ 
z_1 = 1
 \]
(relabelling, if necessary), inequality \eqref{eq:thm-Ap-equiv} is equivalent to the statement that for all $n \geq 2$, $z_2, \dots, z_n \in \C$, we have
\begin{equation}\label{eq:Cp-z1=1}
\E\left[\left|\xi_1 + \sum_{j=2}^{n} z_j\xi_j\right|^{-p}\ \right] \leq C_p\phi_p\left(\sum_{j=2}^n |z_j|^2\right),
 \end{equation}
where
\[ 
\phi_p(x) = (1+x)^{-p/2}, \qquad x \geq 0.
 \]
\begin{remark}\label{rem:Fourier-case} 
Corollary \ref{cor:StepI} establishes this in the special case when
\[ 
|z_2|, \dots, |z_n| \leq 1, \qquad \text{and} \qquad \sum_{j=2}^n |z_j|^2 \geq 1.
 \]
\end{remark}
In order to handle the remaining cases, we follow the approach by Nazarov and Podkorytov from \cite{NP} which is to induct on $n$. However, for \eqref{eq:thm-Ap-equiv} or \eqref{eq:Cp-z1=1} as is, a natural inductive argument does not go through; in essence, the trouble is that $\phi_p$ is convex, whereas one would like it to be concave (for the flavour and more details of the arising difficulties, see \cite[p. 263]{NP}). A compelling idea is to make the assertion stronger, so that the inductive hypothesis is stronger. This is done by considering instead of $\phi_p$, a point-wise smaller function with better-suited concavity-like properties. We set
\begin{equation}\label{eq:def-Phi}
\Phi_p(x) = \begin{cases}
\phi_p(x), & x \geq 1, \\
2\phi_p(1)-\phi_p(2-x), & 0 \leq x \leq 1.
\end{cases}
 \end{equation}

\begin{figure}[htb!]
\centering
\begin{tikzpicture}
  \begin{axis}[
      axis lines = middle,
      axis line style = {->},
 xlabel = $x$,
      ylabel = {$y$},
	  xlabel style = {at={(1,0)}, anchor=west, below=2pt}, 
      ylabel style = {at={(0,1)}, anchor=south east, left=2pt},
      domain = 0:3,
      samples = 300,
      xtick = \empty,
      ytick = \empty,
      grid = none,
      ymin = 0.5,
      ymax = 1.1,
      clip=false,
      legend style = {at={(0.97,0.97)}, anchor=north east, draw=none}
    ]

    \def\p{0.9}

    \pgfmathsetmacro{\yone}{(1+1)^(-\p/2)}     
    \pgfmathsetmacro{\ytwo}{(1+2)^(-\p/2)}     
    \pgfmathsetmacro{\slope}{-(\p/2)*(1+1)^(-\p/2 - 1)} 
    \pgfmathsetmacro{\Phizero}{2*(1+1)^(-\p/2) - (1+2)^(-\p/2)} 
    \pgfmathsetmacro{\ytangentzero}{\slope*(0 - 1) + \yone} 

    \addplot[blue, thick, domain=0:3] {(1+x)^(-\p/2)};
    \addlegendentry{$\phi_p(x)$}

    \addplot[red, thick, domain=0:1]
      {2*(1+1)^(-\p/2) - (1+(2-x))^(-\p/2)};
    \addplot[red, thick, domain=1:3]
      {(1+x)^(-\p/2)};
    \addlegendentry{$\Phi_p(x)$}

    \addplot[blue, mark=*, only marks, mark size=2pt] coordinates {(0,1)};

    \addplot[red, mark=*, only marks, mark size=2pt] coordinates {(0,\Phizero)};

    \addplot[green!60!black, thick, domain=0:1.8]
      {\slope*(x - 1) + \yone};

    \addplot[green!60!black, mark=*, only marks, mark size=2pt]
      coordinates {(0,\ytangentzero) (1,\yone)};

    \addplot[dashed, gray] coordinates {(1,0.5) (1,\yone)};
    \addplot[dashed, gray] coordinates {(2,0.5) (2,\ytwo)};

    \node[below=6pt] at (axis cs:1,0.5) {$x=1$};
    \node[below=6pt] at (axis cs:2,0.5) {$x=2$};
    \node[left] at (axis cs:0,1) {$y=1$};

  \end{axis}
\end{tikzpicture}
\caption{An ingenious construction of $\Phi_p$: the plots show $\phi_p$,~$\Phi_p$ and their common tangent at $x=1$.}
\label{fig:phi}
\end{figure}
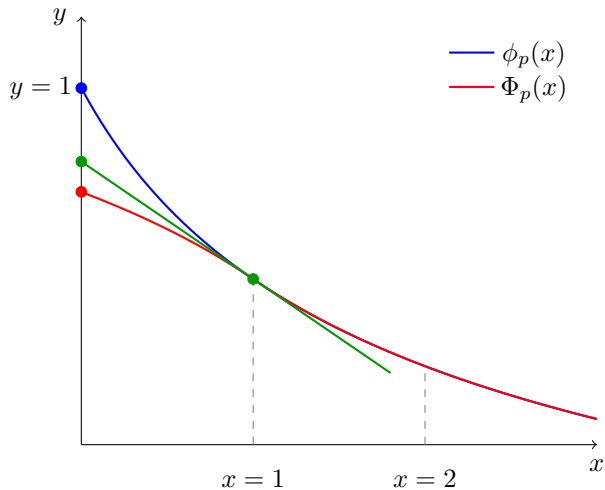

Now, $\Phi_p \leq \phi_p$ (see Figure \ref{fig:phi}) and $\Phi_p$ is plainly concave on $[0,1]$. Additionally, it enjoys the ``extended concavity'' from the next lemma (as called in and known from the earlier papers, put forth in \cite{NP}).

\begin{lemma}[Lemma 33 in \cite{CST}, Lemma 20 in \cite{CKT}]\label{lm:extended-convexity}
Let $p > 0$. For every $a_-, a_+ \geq 0$ with $\frac{1}{2}(a_-+a_+) \leq 1$, we have
\[ 
\frac{\Phi_p(a_-) + \Phi_p(a_+)}{2} \leq \Phi_p\left(\frac{a_- + a_+}{2}\right).
 \]
\end{lemma}

The heart of the argument boils down to proving inductively on $n$ the following theorem which readily gives \eqref{eq:Cp-z1=1}, thus finishing the whole proof.

\begin{theorem}\label{thm:main-ind}
Let $0 < p < 1$ and $\Phi_p$ be defined in \eqref{eq:def-Phi}. For all $n \geq 2$, $z_2, \dots, z_n \in \C$, we have
\begin{equation}\label{eq:main-ind}
\E\left[\left|\xi_1 + \sum_{j=2}^{n} z_j\xi_j\right|^{-p}\ \right] \leq C_p\Phi_p\left(\sum_{j=2}^n |z_j|^2\right).
 \end{equation}
\end{theorem}

\subsection{The technical part: Proof of Lemma \ref{lm:Psi<Psi(2)}}

Define 
\[
F_p(s) = \int_0^\infty \abs{J_0(t)}^st^{p-1} \dd t,
\]
so that $\Psi_p(s) = s^{p/2}F_p(s)$. We begin with bounding $F_p$ by splitting the integration into two regimes, where one has rather accurate pointwise bounds on $J_0$. With hindsight of the ensuing calculations, we set
\begin{equation}\label{eq:def-Up}
U_p(s) = 2^{p-1}s^{-p/2} \left(\Gamma\left(\frac{p}{2}\right) - \frac{\Gamma(\frac{p}{2}+2)}{4s} + \frac{\Gamma(\frac{p}{2}+4)}{32s^2}\right) + \frac{2.59^p (1.295\pi)^{-s/2}}{\frac{s}{2}-p},
\end{equation}
for $s > 2p$.

\begin{lemma}\label{lm:Fp-Up}
For every $0 < p < 1$ and $s > 2p$, we have
\[ 
F_p(s) \leq U_p(s).
 \]
\end{lemma}
\begin{proof}
Using the bound
\[ 
\abs{J_0(t)} \leq \exp\left(-\frac{t^2}{4}-\frac{t^4}{64}\right), \qquad 0 < t < 2.59
 \]
(Proposition 12 in \cite{KK}) near the origin and \eqref{eq:J0-up-bd1} for the remaining range, we obtain
\[
F_p(s) \leq \int_0^{\infty} \exp\left(-\frac{st^2}{4}-\frac{st^4}{64}\right) t^{p-1}\dd t + \int_{2.59}^\infty \left(\sqrt{\frac{2}{\pi t}}\right)^s t^{p-1}\dd t
 \]
The second integral evaluates to the second term in the definition of $U_p$. For the first integral, using the substitution $u=\frac{t^2}{4}$ and the bound $e^{-x}\leq 1-x+\frac{x^2}{2}$, we get
\begin{align*}
    \int_0^\infty \exp\left(-\frac{st^2}{4}-\frac{st^4}{64}\right) t^{p-1}\dd t
    &= \int_0^\infty \exp\left(-u-\frac{u^2}{4s}\right)\left(\frac{4u}{s}\right)^{p/2-1}\frac{2}{s} \dd u\\
    &= 2^{p-1}s^{-p/2} \int_0^\infty e^{-\frac{u^2}{4s}}e^{-u}u^{\frac{p}{2}-1} \dd u\\
    &\leq 2^{p-1}s^{-p/2}\int_0^\infty \left(1-\frac{u^2}{4s}+\frac{u^4}{32s^2}\right) e^{-u} u^{\frac{p}{2}-1}\dd u\\
    &= 2^{p-1}s^{-p/2} \left(\Gamma(\frac{p}{2}) - \frac{\Gamma(\frac{p}{2}+2)}{4s} + \frac{\Gamma(\frac{p}{2}+4)}{32s^2}\right).\qedhere
\end{align*}
\end{proof}

Since the desired inequality $s^{p/2}F_p(s) = \Psi_p(s) \leq \Psi_p(2)$ is clearly equivalent to $F_p(s) \leq s^{-p/2}\Psi_p(2)$, we let
\begin{equation}\label{eq:def-Gp}
G_p(s) = s^{-p/2} \Psi_p(2) = 2^{3p/2-1}s^{-p/2} \Gamma(1-p) \Gamma(p/2) \Gamma(1-p/2)^{-3},
\end{equation}
recalling \eqref{eq:Psi(2)}.
For the sequel, we note at this point also that
\begin{equation}\label{eq:Fp(2)=Gp(2)}
F_p(2) = G_p(2)
 \end{equation}
(since there is equality at $s = 2$).

Continuing to follow the approach from \cite{CST} and \cite{OP}, we establish $U_p(s) \leq G_p(s)$ in as wide range of parameters as possible; alas, given $0 < p < 1$, this does not hold for all $s \geq 2$, but for \emph{all} $s$ sufficiently large, $s \geq 3$. To overcome this deficiency (ultimately of the bound $F_p \leq U_p$), we show a stronger than needed bound at $s=3$ and then interpolate it to \emph{all} $2 \leq s \leq 3$, leveraging H\"older's inequality, i.e. the log-convexity of $F_p(s)$. These steps are the content of the next two lemmas.

\begin{lemma}\label{lm:Up-Gp}
For every $0 < p < 1$ and $s \geq 3$, we have
$
U_p(s) \leq G_p(s).
$
\end{lemma}

\begin{lemma}\label{lm:Fp(3)}
For every $0 < p < 1$, we have $F_p(3) \leq e^{-p/4}G_p(2)$. 
\end{lemma}

Taking the lemmas for granted, we can finish the proof.

\begin{proof}[End of proof of Lemma \ref{lm:Psi<Psi(2)}]
In view of Lemma \ref{lm:Up-Gp}, we need only consider $2 \leq s \leq 3$. We let $\lambda = 3-s \in [0,1]$ and observe that $s = 2\lambda + 3(1-\lambda)$. H\"older's inequality therefore yields
\[ 
F_p(s) = \int_0^\infty \abs{J_0(t)}^{2\lambda + 3(1-\lambda)}t^{p-1} \dd t  \leq F_p(2)^\lambda F_p(3)^{1-\lambda} = G_p(2)^\lambda F_p(3)^{1-\lambda},
 \]
where we have used \eqref{eq:Fp(2)=Gp(2)}. Finally, thanks to Lemma \ref{lm:Fp(3)},
\[ 
F_p(s) \leq G_p(2)^{\lambda} (e^{-p/4}G_p(2))^{1-\lambda} = G_p(2) e^{-p(s-2)/4}.
 \]
Moreover, $e^{-p(s-2)/4} \leq s^{-p/2}2^{p/2}$ (by the concavity of $\log$, see (31) in \cite{CST}), thus
\[ 
\Psi_p(s) = s^{p/2}F_p(s) \leq 2^{p/2}G_p(2) = \Psi_p(2). \qedhere
 \]
\end{proof}

It remains to prove Lemmas \ref{lm:Up-Gp} and \ref{lm:Fp(3)} which occupies the rest of this section (the latter being perhaps the most technical part of this paper).

\begin{proof}[Proof of Lemma \ref{lm:Up-Gp}]
Recalling the definitions of $U_p$ and $G_p$ from \eqref{eq:def-Up} and \eqref{eq:def-Gp}, inequality $U_p \leq G_p$ is equivalent to
\begin{align*}
\frac{(2.59)^p (1.295\pi)^{-s/2}}{\frac{s}{2}-p} &\leq 2^{3p/2-1}s^{-p/2} \Gamma(p/2)\frac{\Gamma(1-p)}{\Gamma(1-p/2)^{3}} \\
&\qquad - 2^{p-1}s^{-p/2} \left(\Gamma\left(\frac{p}{2}\right) - \frac{\Gamma(\frac{p}{2}+2)}{4s} + \frac{\Gamma(\frac{p}{2}+4)}{32s^2}\right).
 \end{align*}
For the ensuing calculations, we set
\[ 
b_0 = 1.295
 \]
and introduce the function
\begin{equation}\label{eq:def-Dp}
D(p) = 2^{p/2}\frac{\Gamma(1-p)}{\Gamma(1-p/2)^{3}}, \qquad p \in (0,1).
\end{equation}
We shall need the following observation, deferring its proof until the main argument has been finished.
\begin{lemma}\label{lm:D(p)}
Function $(0,1) \ni p \mapsto \log D(p)$ is convex, increasing and positive.
\end{lemma}

After multiplying both sides of the previous inequality by $2^{1-p}s^{p/2+2}$ and performing simple algebraic manipulations, it becomes equivalent to
\[ 
\frac{2b_0^p(b_0\pi)^{-s/2}}{\frac{s}{2}-p}s^{p/2+2}\leq s^2\Gamma\left(\frac{p}{2}\right)(D(p)-1) + \Gamma\left(\frac{p}{2}+2\right)\frac{8s-(\frac{p}{2}+2)(\frac{p}{2}+3)}{32}.
 \]

\textbf{Claim.} For every fixed $0 < p < 1$, the left hand side is decreasing in $s \in (3,+\infty)$, whilst the right hand side is increasing in $s$. In particular, it suffices to prove this inequality for $s=3$.

Indeed, for the right hand side, it suffices to note that $D(p) > 1$ (Lemma \ref{lm:D(p)}). For the left hand side, we crudely bound the derivative,
\begin{align*}
\frac{\dd}{\dd s}\log\left( \frac{2b_0^p(b_0\pi)^{-s/2}}{\frac{s}{2}-p}s^{p/2+2} \right) &= -\frac12\log(b_0\pi) + \left(\frac{p}{2}+2\right)\frac{1}{s} - \frac{1}{s-2p} \\
&< -\frac12\log(b_0\pi) + \left(\frac{1}{2}+2\right)\frac{1}{s} - \frac{1}{s} =  -\frac12\log(b_0\pi)  + \frac{3}{2s}
\end{align*}
which is plainly negative for $s > 3$. 

Plugging in $s=3$, and using the fact that $\Gamma(x)\geq 0.88$, $x > 0$ (see \cite{DC}), and $\Gamma(\frac{p}{2}+2) \geq \Gamma(2) = 1$ ($\Gamma(x)$ is increasing for $x> \frac32$), it suffices to prove that the following function is positive for $0 < p < 1$, 
\[
f(p) = 7.92(D(p)-1) + \frac{24-(\frac{p}{2}+2)(\frac{p}{2}+3)}{32}-\frac{2b_0^p(b_0\pi)^{-3/2}}{\frac{3}{2}-p} = R(p) - L(p),
\]
where we have let
\[
R(p) = 7.92(D(p)-1)+\frac{3}{4}, \qquad L(p) = \frac{(\frac{p}{2}+2)(\frac{p}{2}+3)}{32}+\frac{2b_0^p(b_0\pi)^{-3/2}}{\frac{3}{2}-p}.
\]
Observe that $L(p)$ and $R(p)$ are both convex (recall Lemma \ref{lm:D(p)}). 

Take the tangent $l_0(p) = R(0)+pR'(0)$ as a lower bound of $R(p)$. We then numerically confirm that $l_0$ dominates $L$ at $p=0$ and $p=1$ (the difference between $l_0$ and $L$ is greater than $0.4$ at $p=0$ and $0.3$ at $p=1$). This confirms that $R\geq l_0>L$ on $[0,1]$, which finishes the proof. 
\end{proof}

\begin{proof}[Proof of Lemma \ref{lm:D(p)}]
Set $x=\frac{1-p}{2}$, $0<x<1/2$. Then,
    \begin{align*}
        D(p) &= 2^{p/2}\frac{\Gamma(1-p)}{\Gamma(1-p/2)^3} = 2^{1/2-x}\frac{\Gamma(2x)}{\Gamma(x+1/2)^3} = \frac{2^{x-1/2}\Gamma(x)}{\sqrt{\pi}\Gamma(x+\frac12)^2}.
    \end{align*}
Thus, the convexity of $p \mapsto \log D(p)$ on $(0,1)$ is equivalent to the convexity of 
\[
f(x)= \log \Gamma(x) - 2\log\Gamma(x+1/2), \qquad  0 < x < \frac12.
\]
Using the series representation $(\log \Gamma(z))'' = \sum_{n=0}^\infty 1/(n+z)^2$, we see that 
    \begin{align*}
        f''(x) &= \sum_{n=0}^\infty \frac{1}{(x+n)^2} - 2\sum_{n=0}^\infty \frac{1}{(x+n+\frac12)^2}\\
        &= \frac{1}{x^2}-\frac{2}{(x+\frac12)^2}+\sum_{n=1}^\infty \frac{1}{(x+n)^2} - 2\sum_{n=1}^\infty \frac{1}{(x+n+\frac12)^2}.
    \end{align*}
Observe that for $0 < x < \frac12$, 
    \[
\sum_{n=1}^\infty \frac{1}{(x+n)^2} - 2\sum_{n=1}^\infty \frac{1}{(x+n+\frac12)^2} > \sum_{n=1}^\infty \frac{1}{(\frac12+n)^2} - 2\sum_{n=1}^\infty \frac{1}{(n+\frac12)^2} =4 -\frac{\pi^2}{2}.
\]
    Thus 
    \[
f''(x) > \frac{1}{x^2}-\frac{2}{(x+\frac{1}{2})^2}+4-\frac{\pi^2}{2}.
\]
We check that the right hand side is decreasing, thus it is further lower bounded by its value at $x=\frac12$ which is $6-\frac{\pi^2}{2} > 0$. This establishes the convexity of $\log D(p)$. 

Moreover, $D(0)=1$, and $(\log D(p))'|_{p=0} = \frac{\log 2-\gamma}{2}  >0$, so $p\mapsto \log D(p)$ is strictly increasing on $(0,1)$, with $\log D(0) = 0$. 
\end{proof}

\begin{proof}[Proof of Lemma \ref{lm:Fp(3)}]
We adapt and follow rather closely the approach employed in \cite{CST}. We split the integral defining $F_p(3)$ into four integrals, over the consecutive intervals $(0, t_1)$, $(t_1, t_2)$, $(t_2, t_3)$, $(t_3, +\infty)$,
\[ 
F_p(3) = A_1 + \dots + A_4, \qquad A_j = \int_{t_{j-1}}^{t_j} |J_0(t)|^3t^{p-1}\dd t,
 \]
(with the obvious convention $t_0 = 0$, $t_4 = +\infty$). Now we bound each contribution using different methods. With the hindsight of the required numerical calculations, we choose the nodes
\[ 
t_1 = 1, \ t_2 = 3, \ t_3 = 12.
 \]

For $A_1$, we use the elementary pointwise bound
\[ 
\exp\left\{-3\left(\frac{t^2}{4} + \frac{t^4}{64}\right)\right\} \leq 1 - \frac{3t^2}{4} + \frac{15t^4}{64}, \qquad t \geq 0
 \]
which can be justified by the means of Taylor's expansion at $t=0$ with the Lagrange remainder applied to the difference between the right and left hand sides whose derivatives up to the order $4$ vanish at $t=0$ (that is how the right hand side is chosen), and the fifth one turns out to be positive on $0 < t  < 1$. This leads to
\begin{align}\notag
A_1 \leq \int_0^{t_1} \left(1 - \frac{3t^2}{4} + \frac{15t^4}{64}\right) t^{p-1}\dd t &= \frac{1}{p}-\frac{3}{4(p+2)} + \frac{15}{64(p+4)} \\
&< \frac{1}{p}-\frac{3}{4(p+2)} + \frac{15}{256}\label{eq:A1}.
\end{align}

For $A_2$, we tactlessly upper bound $|J_0|^3$ by a piecewise constant function resulting from chopping $(t_1, t_2)$ into $m$ consecutive subintervals of equal length, using $\sup_{[u,v]} |J_0| \leq \max\{|J_0(a), J_0(b)\}$ for every $t_1 = 1 \leq a < b \leq 3 = t_2$ (since $|J_0(t)|$ is first decreasing to its first zero occurring at about $t = 2.40..$ and then increasing to its first local maximum occurring at about $t=3.83..$),
\begin{align}\notag
A_2 \leq \sum_{k=0}^{(t_2-t_1)m-1} \max\left\{\left|J_0\left(1 + \frac{k}{m}\right)\right|^3, \left|J_0\left(1 + \frac{k+1}{m}\right)\right|^3\right\}\int_{t_1+k/m}^{t_1+(k+1)/m} t^{p-1} \dd t \\ 
< \sum_{k=0}^{(t_2-t_1)m-1} \max\left\{\left|J_0\left(t_1 + \frac{k}{m}\right)\right|^3, \left|J_0\left(t_1 + \frac{k+1}{m}\right)\right|^3\right\}\frac{1}{m}\left(t_1+\frac{k}{m}\right)^{p-1}.\label{eq:A2}
\end{align}

For $A_3$, quite similarly, but due to the lack of monotonicity, we use Riemann sums (with midpoints) and control the error crudely bounding the derivative, using the elementary scheme $\int_a^b h(t)t^{p-1}\dd t \leq h(\frac{a+b}{2})\int_a^b t^{p-1}\dd t + \sup_{[a,b]}|h'|\frac{1}{2(b-a)}\int_a^bt^{p-1}\dd t$. We arrive at
\begin{align}\notag
A_3 \leq \sum_{k=0}^{(t_3-t_2)m-1} \left|J_0\left(t_2 + \frac{k+\frac12}{m}\right)\right|^3\int_{t_2+k/m}^{t_2+(k+1)/m} t^{p-1} \dd t + L\frac{1}{2m}\int_{t_2}^{t_3} t^{p-1}\dd t \\ 
< \sum_{k=0}^{(t_3-t_2)m-1} \left|J_0\left(t_2 + \frac{k+\frac12}{m}\right)\right|^3\frac{1}{m}\left(t_2+\frac{k}{m}\right)^{p-1} + \frac{L(t_3-t_2)t_2^{p-1}}{2m}, \label{eq:A3}
\end{align}
where
\[ 
L = \sup_{t \in [t_2, t_3]}\left|\frac{\dd}{\dd t} |J_0(t)|^3\right| < 0.1.
 \]
Since $\left|\frac{\dd}{\dd t} |J_0(t)|^3\right| = 3|J_0(t)|^2|J_1(t)|$, this is directly verified numerically (there are concrete precise approximations to $J_0$ and $J_1$, see e.g. 9.4.3 in \cite{AS}).

Finally, for $A_4$, we simply use \eqref{eq:J0-up-bd1} which on $(t_3, +\infty)$ turns out to be accurate enough,
\begin{align}\label{eq:A4}
A_4 \leq \int_{t_3}^\infty \left(\frac{2}{\pi t}\right)^{3/2}t^{p-1}\dd t = \left(\frac{2}{\pi }\right)^{3/2} \frac{t_3^{p-3/2}}{3/2 - p}.
\end{align}
With hindsight, we choose
\[ 
m = 100.
 \]

Let us denote the bounds from \eqref{eq:A1}--\eqref{eq:A4} by $B_1(p), \dots, B_4(p)$ and let
\begin{align*}
L(p) &= p(B_1(p) + \dots + B_4(p)), \\
R(p) &= pe^{-p/4}G_p(2) = pe^{-p/4}2^{p-1}\Gamma(1-p) \Gamma(p/2) \Gamma(1-p/2)^{-3} \\
&= e^{-p/4}2^{p/2}\Gamma(1+p/2)D(p),
 \end{align*}
as per \eqref{eq:def-Gp} and \eqref{eq:def-Dp}. It is straightforward to argue that $L(p)$ is convex (in fact, each $pB_j(p)$, $j = 1, \dots, 4$ is convex, e.g. by directly examining the derivatives of $pB_1(p)$ and $pB_4(p)$, as well as noting that $pB_2(p)$ and $pB_3(p)$ are of the form $\sum \lambda_i p a_i^p$ with constants $\lambda_i \geq 0$, $a_i \geq 1$). Plainly, $L(0) = R(0) = 1$. For $0 < p \leq 0.02$, we use the tangent $\ell_0(p) = R(0) + R'(0)p$ to lower bound $R(p)$, and calculate  that $\ell_0(p) - L(p) > 10^{-5}$, to obtain $R(p) \geq \ell_0(p) \geq L(p)$, $0 < p \leq 0.02$. 
To finish, we divide the remaining interval $(0.02, 1)$ into $(u_j, u_{j+1})$, $1 \leq j \leq 5$ using nodes 
\[
u_1 = 0.02, u_2 = 0.06, u_3 = 0.15, u_4 = 0.3, u_5 = 0.7, u_6 = 1.
\]
On each segment $[u_j, u_{j+1}]$, we lower bound $R(p)$ using its tangent at the starting point, $\ell_j(p) = R(u_j) + R'(u_j)(p-u_j)$, and then compare it against $L(p)$ looking at the end-points, that is verify that the differences $d_-(j) = \ell_j(u_j) - L(u_j)$ and $d_+(j) = \ell_j(u_{j+1}) - L(u_{j+1})$ are positive, which gives $R(p) > L(p)$ for all $u_j \leq p \leq u_{j+1}$. The corresponding numerical bounds are collected in Table \ref{tab:d}. This finishes the proof.
\end{proof}
\begin{table}[h!]
  \begin{center}
    \label{tab:d}
    \begin{tabular}{c|c|c|c|c|c} 
      $j$ & 1 & 2 & 3 & 4 & 5 \\
      \hline
      $d_-(j)$ & $0.00017$ & $0.0008$ & $0.004$ & $0.02$ & $0.28$ \\
      $d_+(j)$ & $0.00017$  & $0.0006$ & $0.06$ & $0.003$ & $0.7$ \\
    \end{tabular}
\vspace*{1em}
    \caption{Proof of Lemma \ref{lm:Fp(3)}: Lower bounds on the differences $d_{\pm}(j)$.}
  \end{center}
\end{table}

\subsection{The inductive argument: Proof of Theorem \ref{thm:main-ind}}\label{sec:ind}

We follow closely the scheme of \cite{NP}. For Steinhaus random variables, working with complex coefficients makes the required algebraic manipulations particularly natural and suggestive. As a side note: their adaptations to higher-dimensional settings as well as the uniform distribution on the real line have been developed in \cite{CKT, CGT, CST}. 

We begin the inductive argument. For the base case $n=2$, we have the following claim, the (easy) proof of which we defer for now.

\textbf{Claim.} $\E|\xi_1 + z\xi_2|^{-p} \leq C_p\Phi_p(|z|^2)$, $z \in \C$.

Let $n \geq 3$ and suppose \eqref{eq:main-ind} holds for every sequence $(z_2, \dots)$ of complex numbers of length $n-1$. Let $z_2, \dots, z_n \in \C$ and 
\[
x = |z_2|^2+\dots+|z_n|^2.
\]
We want to show \eqref{eq:main-ind}. There are 3 cases.

Case (a): $|z_k| > 1$ for some $2 \leq k \leq n$. Then $x > 1$, so by the definition of $\Phi_p$, \eqref{eq:main-ind} coincides with
\begin{equation}\label{eq:K''}
\E\left|\sum_{k=1}^n z_k\xi_k\right|^{-p} \leq C_p(|z_1|^2+|z_2|^2+\ldots+|z_n|^2)^{p/2},
\end{equation}
where $z_1 = 1$.
Let $z_1^\ast,\ldots,z_n^\ast$ be a rearrangement of $z_1,\ldots,z_n$ such that $|z_k^\ast|\geq |z_{k+1}^\ast|$ for every $k=1,\ldots,n-1$ and let $z'_k = \frac{z_k^\ast}{|z_1^\ast|}$ for every $k=1,\ldots,n$, so that $|z'_1|=1$ and $|z'_k|\leq 1$ for $k=2,\ldots,n$. Then due to the homogeneity of \eqref{eq:K''}, it is enough to prove
\[
\E\left|\xi_1 + \sum_{k=2}^n z'_k\xi_k\right|^{-p} \leq C_p\Phi_p(|z'_2|^2+\dots+|z'_n|^2),
\]
which is handled by the next cases.

Case (b): $|z_k| \leq 1$ for every $2 \leq k \leq n$ and $x \geq 1$. This is handled by the Fourier part, recall Remark \ref{rem:Fourier-case}.

Case (c): $|z_k| \leq 1$ for every $2 \leq k \leq n$ and $x < 1$. Since $(\xi_{n-1},\xi_n)$ has the same distribution as $(\xi_{n-1},\xi_{n-1}\xi_n)$, we have,
\begin{align*}
\E\left|\xi_1 + \sum_{k=2}^n z_k\xi_k\right|^{-p} &= \E\left|\xi_1 + z_2\xi_2 + \dots + z_{n-1}\xi_{n-1}+ z_n\xi_n\xi_{n-1}\right|^{-p} \\
&= \E_{\xi_n}\left[\E_{(\xi_k)_{k=2}^{n-1}}\left|\xi_1 +  z_2\xi_2 + \dots + z_{n-2}\xi_{n-2} + (z_{n-1}+z_n\xi_n)\xi_{n-1} \right|^{-p}\right].
\end{align*}
By the inductive hypothesis applied to the sequence $(z_2, \dots, z_{n-2}, z_{n-1}+z_n\xi_n)$ (conditioned on the value of $\xi_n$), we get
\[
\E\left|\xi_1 + \sum_{k=2}^n z_k\xi_k\right|^{-p} \leq C_p\E_{\xi_n}\Phi_p(|z_2|^2 + \dots + |z_{n-2}|^2 + |z_{n-1}+z_n\xi_n|^2).
\]
Note that $|z_2|^2 + \dots + |z_{n-2}|^2 + |z_{n-1}+z_n\xi_n|^2 = x + 2\text{Re}(z_{n-1}\overline{z_n\xi_n})$, so by the symmetry of $\xi_n$ and the extended convexity of $\Phi_p$  from Lemma \ref{lm:extended-convexity} applied to $x' = x + 2\text{Re}(z_{n-1}\overline{z_n\xi_n})$ and $x'' = x - 2\text{Re}(z_{n-1}\overline{z_n\xi_n})$ which satisfy $x', x'' \geq 0$ and $\frac{x'+x''}{2} = x < 1$, we obtain
\begin{align*}
&\E_{\xi_n}\Phi_p(|z_2|^2 + \dots + |z_{n-2}|^2 + |z_{n-1}+z_n\xi_n|^2) \\
&= \E_{\xi_n}\frac{\Phi_p(x + 2\text{Re}(z_{n-1}\overline{z_n\xi_n})) + \Phi_p(x - 2\text{Re}(z_{n-1}\overline{z_n\xi_n}))}{2} \leq \Phi_p(x),
\end{align*}
which finishes the proof of the inductive step.

It remains to show the claim. 

\begin{proof}[Proof of the claim.]
Letting $x = |z|^2$, equivalently, the claim is that $\E|\xi_1 + \sqrt{x}\xi_2|^{-p} \leq C_p\Phi_p(x)$ for every $x \geq 0$. 

When $x \geq 1$, $\Phi_p(x) = \phi_p(x) = (1+x)^{-p/2}$. Multiplying both sides of the desired inequality by $x^{p/2}$ and leveraging its homogeneity, it becomes $\E|\sqrt{1/x}\xi_1 + \xi_2|^{-p} \leq C_p\phi_p(1/x)$. Since $\phi_p \leq \Phi_p$, this case will follow from the case $x \leq 1$.

When $0 < x < 1$, we have
\begin{align*}
\E|\xi_1 + \sqrt{x}\xi_2|^{-p} &= \E|1 + \sqrt{x}\xi_2|^{-p/2} =  \frac{1}{2\pi}\int_0^{2\pi}\left|1+\sqrt{x}e^{it}\right|^{-p}\dd t \\
&= \frac{1}{2\pi}\int_0^{2\pi}\left(1+\sqrt{x}e^{it}\right)^{-p/2}\left(1+\sqrt{x}e^{-it}\right)^{-p/2}\dd t \\
&=  \sum_{k, l=0}^\infty \binom{-p/2}{k}\binom{-p/2}{l}x^{(k+l)/2}\frac{1}{2\pi}\int_0^{2\pi} e^{it(k-l)}\dd t \\
&= \sum_{k=0}^\infty \binom{-p/2}{k}^2x^k,
 \end{align*}
see also Lemma 13 in \cite{CST}, or (27) in \cite{Ko}. In particular, the left hand side of the claimed inequality is increasing in $x$. Since $\Phi_p(x)$ is decreasing, it suffices to check that the inequality holds at $x=1$, which becomes equality by the definition of~$C_p$.
\end{proof}

\section{A comment on \cite{Ko}}\label{sec:Ko}
As mentioned in the introduction, the main result of \cite{Ko} is that for $0 < p < 1$ and every $z_1, \dots, z_n \in \C$, we have
\begin{equation}\label{eq:K}
\E|z_1\xi_1+\dots+z_n\xi_n|^p \geq A_p^p\left(|z_1|^2+\dots+|z_n|^2\right)^{p/2}.
\end{equation}
with $A_p$ defined in \eqref{eq:Ap-Bp-sharp}. This result is correct, however its proof contains an error.

Instead of the ``global'' inductive argument (as the one employed above), the proof of \eqref{eq:K} in \cite{Ko} aims at showing by induction on $n$ that a strengthened version of \eqref{eq:K} -- Claim (29) in \cite{Ko} -- holds for all sequences $z= (z_1, \dots, z_n)$ with $\|z||_\infty \geq \frac{1}{\sqrt2}\|z\|_2$ (with the opposite case handled independently via Fourier-analytic arguments). The argument proceeds  under the normalisation $|z_1| = \|z\|_\infty = 1$, so for all sequences $(z_2, \dots, z_n)$ with $\sum_{k=2}^n |z_k|^2 \leq 1$, and the inductive hypothesis is applied to sequences $(z_2, \dots, z_{n-2}, z_{n-1}+\xi z_n)$ with a \emph{fixed} arbitrary $\xi \in \C$, $|\xi| = 1$  (line -5 on p. 52). This is erroneous because such sequences in general may fail to satisfy the assumption $\sum_{k=2}^n |z_k|^2 \leq 1$. This problem seems unavoidable, resulting from an earlier quite loose bound of an average $\E h(\xi_1)$ by an infimal value $\inf_{|\xi_1|=1} h(\xi_1)$ (l.14 on p.52).

One can of course fix the issue by following verbatim the ``global'' inductive argument from Section \ref{sec:ind} above (with Case (b) -- the Fourier part -- established in \cite{Ko}, as well as all the other relevant claims, like the base case for induction, or the extended convexity).

\section{Sharp R\'enyi entropy comparison}\label{sec:Renyi}

For a continuous complex-valued random variable $X$ with density $f$ on $\mathbb{C}$, the
Rényi entropy of order $p \in [0,\infty]$ (put forward in \cite{Ren}) is defined by
\[
    h_p(X)=h_p(f)
    = \frac{1}{1-p} \log\!\left( \int_{\mathbb{C}} f(x)^p \, dx \right),
\]
with the cases $p \in \{0, 1, \infty\}$ understood by usual limiting expressions,
\[
    h_0(f)= \log \mathrm{Leb}\Big(\mathrm{supp}(f)\Big), \qquad h(f) = h_1(f) = - \int_{\mathbb{C}} f \log f, \qquad 
    h_\infty(f)= - \log \|f\|_\infty,
\]
(provided the integrals exist; here $ \mathrm{Leb}\Big(\mathrm{supp}(f)\Big)$ denotes the Lebesgue measure of the support of $f$  and $\|f\|_\infty$ denotes the essential supremum of $f$). Akin to sharp moment comparison inequalities, it has been of interest to investigate analogous entropy bounds, particularly under the fixed variance constraint, in a variety of contexts (discrete, probabilistic, information theoretic,  geometric, and beyond), see \cite{MMX}. We refer to the comprehensive introduction in \cite{BNZ} for motivations as well as further references, and mention here only in passing several other works \cite{BN, BNT, CST-gamma, ENT-GM, MNR}. 

To parallel a result obtained recently for the uniform distribution in \cite{CGT}, we record the following sharp upper bound.

\begin{theorem}\label{thm:Renyi}
Let $p \in [0,1]$. For every unit vector $a = (a_1, \dots, a_n)$ in $\R^n$, we have 
\[ 
h_p\left(\sum_{j=1}^n a_j\xi_j\right) \leq h_p(Z),
 \]
where $Z$ is the standard complex Gaussian with mean $0$ and covariance $\tfrac12 I_{2\times2}$, so that $\E|Z|^2 = 1$.
\end{theorem}

The proof of this result follows verbatim the proof of Theorem~2 from \cite{CGT}. We shall not repeat it here, but only recall that it relies on the sharp upper bound from  \eqref{eq:Ap-Bp} for even moments, $p = 2, 4, 6, \dots$ which holds with the sharp Gaussian constant $B_p = \|Z\|_p$, \eqref{eq:Ap-Bp-sharp}. We also remark that since this continues to hold for $\xi_j$ uniform on Euclidean spheres in all dimensions (see \cite{BC}), Theorem \ref{thm:Renyi} automatically extends to that case as well. 

\begin{remark}\label{rem:ent-no-lower-bd}
Since the distribution of $\xi_1$ is singular on $\C$ (as supported on the set of measure $0$), $h_p(\xi_1) = -\infty$. Consequently, for Steinhaus sums $S$ of fixed variance, there is no nontrivial lower bound on their $p$-R\'enyi entropy. This contrasts the case of convolutions of any continuous distributions of finite entropy, for which we can obtain a lower-bound directly from the entropy power inequality (for instance, see Theorem 2 in \cite{CGT}). We find it of interest to devise meaningful lower bounds for Steinhaus sums $S = \sum_{j=1}^n a_j\xi_j$ when one precludes $S = \xi_1$, for example by constraining the $\ell_\infty$ norm of the coefficient vector $a$, say imposing $\|a\|_\infty < 1 - \delta$ for some small constant $\delta > 0$, provided that $\|a\|_2 = 1$. We leave it for further investigations.
\end{remark}



\end{document}